\documentclass[11pt]{article}
\usepackage[utf8]{inputenc}
\usepackage[english]{babel}
\usepackage{amsmath}
\usepackage{amssymb}
\usepackage{amsthm}
\usepackage{verbatim}
\usepackage{graphicx}
\usepackage{subcaption}
\usepackage[shortlabels]{enumitem}
\usepackage{xcolor}
\usepackage{float}
\usepackage[pagewise]{lineno}

\newtheorem{theorem}{Theorem}[section]
\newtheorem{corollary}{Corollary}[theorem]

\providecommand{\keywords}[1]{{\textit{Keywords:}} #1}

\newcommand{\VNN}{V_{00}}

\newcommand{\KK}{{\mathsf{K_h}}}
\newcommand{\LKK}{{\mathsf{K}}}
\newcommand{\WW}{{\mathsf{W}}}
\newcommand{\PP}{{\mathsf{P}}}
\newcommand{\PPN}{{\mathsf{P^{\mathcal{N}}\! \!}}}

\newcommand{\xx}{{\mathbf{x}}}
\newcommand{\yy}{{\mathbf{y}}}
\newcommand{\zz}{{\mathbf{z}}}
\newcommand{\bb}{{\mathbf{b}}}

\newcommand{\nn}{{\mathbf{n}}}

\newcommand{\NK}{\mathcal{N}(\LKK)}

\newcommand{\pphi}{{\mathbf{\phi}}}

\newcommand{\nullspace}[1]{{\mathcal{N}(#1)}}

\DeclareMathOperator*{\argmin}{arg\,min}
\DeclareMathOperator*{\argmax}{arg\,max}

\title{Modified Tikhonov regularization for identifying several sources}
\author{Ole L{\o}seth Elvetun\thanks{Faculty of Science and Technology, Norwegian University of Life Sciences, P.O. Box 5003, NO-1432 {\AA}s, Norway. Email: ole.elvetun@nmbu.no.} and Bj{\o}rn Fredrik Nielsen\thanks{Faculty of Science and Technology, Norwegian University of Life Sciences, P.O. Box 5003, NO-1432 {\AA}s, Norway. Email: bjorn.f.nielsen@nmbu.no. Nielsen's work was supported by The Research Council of Norway, project number 239070.}}

\begin{document}

\maketitle

\begin{abstract}
We study whether a modified version of Tikhonov regularization can be used to identify several local sources from Dirichlet boundary data for a prototypical elliptic PDE. This paper extends the results presented in \cite{Elv20}.
It turns out that the possibility of distinguishing between two, or more, sources depends on the smoothing properties of a second or fourth order PDE.
Consequently, the geometry of the involved domain, as well as the position of the sources relative to the boundary of this domain, determines the identifiability. 

We also present a uniqueness result for the identification of a single local source. This result is derived in terms of an abstract operator framework and is therefore not only applicable to the model problem studied in this paper. 

Our schemes yield quadratic optimization problems and can thus be solved with standard software tools. In addition to a theoretical investigation, this paper also contains several numerical experiments. 
\end{abstract}

\section{Introduction}
\label{section:introduction}
We will study the following problem: 
    \begin{equation} \label{eq1}
        \min_{(f,u) \in F_h \times H^1(\Omega)} \left\{ \frac{1}{2}\|u-d\|_{L^2(\partial\Omega)}^2 + \frac{1}{2}\alpha\|\WW f\|_{L^2(\Omega)}^2 \right\}
    \end{equation}
    subject to 
    \begin{equation} \label{eq2}
    \begin{split}
        -\Delta u + \epsilon u &= f \quad \mbox{in } \Omega, \\
        \frac{\partial u}{\partial \nn} &= 0  \quad \mbox{on } \partial \Omega, 
    \end{split}
    \end{equation}
    where $F_h$ is a finite dimensional subspace of $L^2(\Omega)$, $\WW: F_h \rightarrow F_h$ is a linear  regularization operator, $\alpha$ is a regularization parameter, $d$ represents Dirichlet boundary data, $\epsilon$ is a positive constant, $\nn$ denotes the outwards pointing unit normal vector of the boundary $\partial \Omega$ of the bounded domain $\Omega$, and $f$ is the source. Depending on the choice of $\WW$, we obtain different regularization terms, including the standard version $\WW = I$ (the identity map). 
    
    The purpose of solving \eqref{eq1}-\eqref{eq2} is to estimate the unknown source $f$ from the Dirichlet boundary data $u=d$ on $\partial \Omega$. Mathematical problems similar to this occur in numerous applications, e.g., in EEG investigations and in the inverse ECG problem, and has been studied by many scientists, see, e.g., \cite{Abdel15,babda09,cheng15,Bad00, Han11,Het96,hinze19,BIsa05,kun94,ring95,song12,Wan17,Zha18,Zha19}. A more detailed description of previous investigations is presented in \cite{Elv20}. 
    
    In \cite{Elv20} we showed with mathematical rigor that a particular choice of $\WW$ {\em almost} enables the identification of the position of a {\em single} local source from the boundary data. That paper also contains numerical experiments suggesting that two or three local sources, in some cases, can be recovered. The purpose of this paper is to explore the {\em several sources} situation in more detail, both theoretically and experimentally. Moreover, we prove that our particular choice of $\WW$, which will be presented below, enables the {\em precise} recovery of a single local source.   
    
\section{Analysis} 
\subsection{Results for general problems}
Let us consider the abstract operator equation 
\begin{equation}
\label{A1}
\KK \xx = \bb, 
\end{equation}
where $\KK: X \rightarrow Y$ is a linear operator with a {\em nontrivial nullspace} and possibly very small singular values, $X$ and $Y$ are real Hilbert spaces, $X$ is finite dimensional and $\bb \in Y$. (For the problem \eqref{eq1}-\eqref{eq2}, $\KK$ is the forward operator 
$$
\KK: F_h \rightarrow L^2(\partial \Omega) , \quad
f \mapsto u|_{\partial \Omega},
$$
where $F_h$ is a finite dimensional subspace of $L^2(\Omega)$, and $u$ is the unique solution of the boundary value problem \eqref{eq2} for a given $f$.) 

Applying traditional Tikhonov regularization yields the approximation 
\begin{equation}
\label{A0}
\xx_{\alpha} = \argmin_{\xx} \left\{ \frac{1}{2} \| \KK \xx - \bb \|_Y^2 + \frac{1}{2} \alpha \| \xx \|_X^2 \right\}, 
\end{equation}
and, according to standard theory, the minimum norm least squares solution $\xx^*$ of \eqref{A1} satisfies 
\[
\xx^* = \lim_{\alpha \rightarrow 0} \xx_{\alpha} = \KK^{\dagger} \bb \in \nullspace{\KK}^{\perp}, 
\]
where $\nullspace{\KK}^{\perp}$ denotes the orthogonal complement of the nullspace $\nullspace{\KK}$ of $\KK$, and $\KK^{\dagger}$ represents the Moore-Penrose inverse of $\KK$.  

Throughout this paper we assume that $$\mathcal{B} = \{ \pphi_1, \, \pphi_2, \, \ldots, \pphi_n \}$$ is an {\em orthonormal basis} for $X$ and that 
\begin{equation}
    \label{A2}
    K_h(\pphi_i) \neq c K_h(\pphi_j) \quad \mbox{for } i \neq j \mbox{ and } c \in \mathbb{R}.
\end{equation} 
That is, the images under $\KK$ of the basis functions are not allowed to be parallel. Note that \eqref{A2} asserts that none of the basis functions belong to the nullspace $\nullspace{\KK}$ of $\KK$. (For PDE-constrained optimization problems one can, e.g., choose basis functions with local support. We will return to this matter in subsection \ref{elliptic_source_problems}.)   

Throughout this text, 
\begin{equation}
\label{A3.02} 
\PP: X \rightarrow \nullspace{\KK}^{\perp}
\end{equation}
denotes the orthogonal projection of elements in $X$ onto $\nullspace{\KK}^{\perp}$.
In \cite{Elv20} we investigated whether a single basis function $\pphi_j$ can be recovered from its image\footnote{Since $\KK$ has a nontrivial nullspace, it is by no means obvious that $\pphi_j$ can be recovered from its image $\KK(\pphi_j)$.} $\KK \pphi_j$. More specifically, using the fact that $\KK^{\dagger} \KK = \PP$, we observe that the minimum norm least squares solution $\xx_j^*$ of 
\begin{equation}
\label{A3}
\KK \xx = \KK \pphi_j
\end{equation} 
is 
\begin{equation}
\label{A3.01}
\xx_j^* = \KK^{\dagger} (\KK \pphi_j) = \PP \pphi_j. 
\end{equation}
Furthermore, provided that the linear regularization operator $\WW: X \rightarrow X$ is defined by 
\begin{equation}
    \label{A3.1}
    \WW \pphi_i = \| \PP \pphi_i \|_X \pphi_i \quad \mbox{for } i=1,2,\ldots, n, 
\end{equation}
 it follows from \eqref{A3.01}, the orthonormality of the basis $\mathcal{B}= \left\{ \pphi_1, \pphi_2, \ldots, \pphi_n \right\}$ and basic properties of orthogonal projections that 
\begin{align}
\nonumber
\WW^{-1} \xx_j^* &= \WW^{-1} \PP \pphi_j \\
\nonumber
&=\WW^{-1} \sum_{i=1}^{n} \left(\PP \pphi_j, \pphi_i \right)_X  \pphi_i \\
\nonumber
&=\WW^{-1} \sum_{i=1}^{n} \left(\PP \pphi_j, \PP \pphi_i \right)_X  \pphi_i \\
\nonumber
&=\sum_{i=1}^{n} \left( \PP \pphi_j, \frac{\PP \pphi_i}{\| \PP \pphi_i \|_X} \right)_X  \pphi_i \\
\label{A3.2}
&= \| \PP \pphi_j \|_X \sum_{i=1}^{n} \left( \frac{\PP \pphi_j}{\| \PP \pphi_j \|_X}, \frac{\PP \pphi_i}{\| \PP \pphi_i \|_X} \right)_X  \pphi_i.  
\end{align}
Consequently, the minimum norm least squares solution $\xx_j^*$ of \eqref{A3} is such that 
\begin{equation}
\label{A4}
j \in \argmax_{i \in \{1,2, \ldots, n\}} \left( \WW^{-1} \xx_j^* (i) \right),
\end{equation}
where $\WW^{-1} \xx_j^* (i)$ denotes the $i$'th component of the Euclidean vector $[\WW^{-1} \xx_j^*]_{\mathcal{B}} \in \mathbb{R}^n$. This implies that we almost can recover the basis function $\pphi_j$ from its image $\KK \pphi_j$: Compute the minimum norm least squares solution $\xx_j^*$ of \eqref{A3}. Then $j$ is among the indexes for which $\WW^{-1} \xx_j^*$ attains its maximum. For further details, see Theorem 4.2 in \cite{Elv20}. We write {\em almost} because the maximum component of $[\WW^{-1} \xx_j^*]_{\mathcal{B}}$ may not be unique.

Based on these findings, we defined {\bf \underline{Method I}}  in \cite{Elv20} as: Compute 
\begin{equation}
\label{A4.1}
\WW^{-1} \xx_{\alpha}, 
\end{equation}
where $\xx_{\alpha}$ is the outcome of applying standard Tikhonov regularization \eqref{A0}, and the operator $\WW$ is defined in \eqref{A3.1}. (Assume that $\pphi_j$ is a basis function with local support. Then the discussion above shows that a local source equaling $\pphi_j$ almost can be recovered by Method I from its image $\KK \pphi_j$.) 

\subsubsection{Uniqueness}
We will now show that we can replace "$\in$" in \eqref{A4} with equality if \eqref{A2} holds, i.e., the maximum component of $[\WW^{-1} \xx_j^*]_{\mathcal{B}}$ is unique. 
\begin{theorem}
\label{theorem:uniqueness}
Assume that the basis $\mathcal{B}$ for $X$ is orthonormal and that \eqref{A2} holds. Then the minimum norm least squares solution $\xx_j^*$ of \eqref{A3} is such that 
\begin{equation}
\nonumber
j = \argmax_{i \in \{1,2, \ldots, n\}} \left( \WW^{-1} \xx_j^* (i) \right),
\end{equation}
where $\WW$ is defined in \eqref{A3.1}.
\end{theorem}
\begin{proof}
Recall the expression \eqref{A3.2} for $\WW^{-1} \xx_j^*$ and the definition \eqref{A3.02} of the projection $\PP$. From the Cauchy-Schwarz inequality we know that 
\[
\left| \left( \frac{\PP \pphi_j}{\| \PP \pphi_j \|_X}, \frac{\PP \pphi_i}{\| \PP \pphi_i \|_X} \right)_X \right| \leq 1, 
\]
with equality if, and only if, there is a constant $c$ such that $\PP \pphi_i = c \PP \pphi_j$. 

Let 
\begin{equation}
\label{A4.2}
\PPN: X \rightarrow \nullspace{\KK}
\end{equation}
denote the orthogonal projection mapping elements of $X$ onto $\nullspace{\KK}$, i.e., $$\PPN=I-\PP,$$ where $I$ is the identity mapping. Assume that $\PP \pphi_i = c \PP \pphi_j$ for some $c \in \mathbb{R}$. The orthogonal decomposition 
\[
\pphi_i=\PP \pphi_i + \PPN \pphi_i
\]
yields that 
\begin{align*}
\KK \pphi_i &= \KK(\PP \pphi_i) + \KK(\PPN \pphi_i) \\
&=\KK(\PP \pphi_i) \\
&=\KK(c \PP\pphi_j) \\
&=c \KK(\PP \pphi_j) \\
&=c \KK(\PP \pphi_j + \PPN \pphi_j) \\
&=c \KK(\pphi_j).
\end{align*}
We have thus proved the implication 
\begin{equation}
\label{A5}
\PP \pphi_i = c \PP \pphi_j \quad \Rightarrow \quad \KK \pphi_i = c \KK \pphi_j. 
\end{equation}
Consequently, if \eqref{A2} holds, then 
\[
\left| \left( \frac{\PP \pphi_j}{\| \PP \pphi_j \|_X}, \frac{\PP \pphi_i}{\| \PP \pphi_i \|_X} \right)_X \right| < 1 \quad \mbox{for } i \neq j.
\]
The result now follows from \eqref{A3.2}.
\end{proof}

This result only shows that we can recover the individual basis function $\pphi_j$ from its image $\KK \pphi_j$. Nevertheless, the numerical experiments in \cite{Elv20} indicate that Method I also is capable of identifying more general local sources from Dirichlet boundary data. We will discuss this issue in more detail in subsection \ref{composite_sources} below. 

\subsubsection*{Remark} We mention that the opposite implication of \eqref{A5} also holds: Assume that $\KK \pphi_i = c \KK \pphi_j$. Then, see \eqref{A3.01}, 
\begin{align*}
  \PP \pphi_i &= \KK^{\dagger}(\KK \pphi_i) \\
  &= \KK^{\dagger}(c\KK \pphi_j) \\
  &=c \KK^{\dagger}(\KK \pphi_j) \\
  &= c \PP \pphi_j,
\end{align*}
which together with \eqref{A5} leads to the conclusion   
\[
\PP \pphi_i = c \PP \pphi_j \quad \Longleftrightarrow \quad \KK \pphi_i = c \KK \pphi_j.
\]

\subsubsection{Several sources}
Since Theorem \ref{theorem:uniqueness} asserts that the maximum component of $[\WW^{-1} \xx_j^*]_{\mathcal{B}}$ is unique, it makes sense to use the linearity of the problem to extend Theorem \ref{theorem:uniqueness} to cases involving several basis functions:  
\begin{corollary}
\label{corollary:several_source}
Let $\{j_1, \, j_2, \ldots, \, j_r \} \subset \{ 1,2,\ldots,n \}$ be an index set and assume that \eqref{A2} holds. Then the minimum norm least squares solution $\xx^*$ of 
\begin{equation}
\label{A6.01}
\KK \xx = \KK (\pphi_{j_1}+\pphi_{j_2}+ \cdots +\pphi_{j_r})
\end{equation} 
satisfies  
\begin{equation}
\label{A6.1}
    \WW^{-1}\xx^* = \WW^{-1}\xx_{j_1}^* 
    +\WW^{-1}\xx_{j_2}^* 
    +\cdots  
    + \WW^{-1}\xx_{j_r}^*, 
\end{equation}
where 
\begin{align}
\nonumber
&j_q = \argmax_{i \in \{1,2, \ldots, n\}} \left( \WW^{-1} \xx_{j_q}^* (i) \right), \\ 
\label{A7}
&\WW^{-1} \xx_{j_q}^* = \| \PP \pphi_{j_q} \|_X \sum_{i=1}^{n} \left( \frac{\PP \pphi_{j_q}}{\| \PP \pphi_{j_q} \|_X}, \frac{\PP \pphi_i}{\| \PP \pphi_i \|_X} \right)_X  \pphi_i 
\end{align}
for $q=1,2,\ldots,r$. Here, $\xx_{j_q}^*$ denotes the minimum norm least squares solution of 
\begin{equation}
\nonumber
\KK \xx = \KK \pphi_{j_q} \quad \mbox{for } q=1,2,\ldots,r.
\end{equation} 
\end{corollary}
\begin{proof}
Since $\xx_{j_q}^* = \KK^{\dagger}(\KK \pphi_{j_q})$, the linearity of $\KK$, $\KK^{\dagger}$ and $\WW$ imply that 
\begin{align*}
  \WW^{-1} \xx^* &= \WW^{-1} \KK^{\dagger}(\KK \pphi_{j_1})+\WW^{-1}\KK^{\dagger}(\KK \pphi_{j_2})+\cdots+\WW^{-1}\KK^{\dagger}(\KK \pphi_{j_r}) \\
  &= \WW^{-1} \xx_{j_1}^* +\WW^{-1} \xx_{j_2}^*+\cdots+\WW^{-1}\xx_{j_r}^*, 
\end{align*}
and the result therefore follows from Theorem \ref{theorem:uniqueness} and \eqref{A3.2}.
\end{proof}

Roughly speaking, Corollary \ref{corollary:several_source} shows that $\WW^{-1}\xx^*$ can be written as a sum \eqref{A6.1} of vectors which achieve their maximums for the correct indices. Consequently, 
if the index subsets associated with the significantly sized components of the Euclidean vectors 
\begin{equation}
\label{A8}
[\WW^{-1}\xx_{j_1}^*]_{\mathcal{B}}, \, [\WW^{-1}\xx_{j_2}^*]_{\mathcal{B}}, \, \ldots, \, [\WW^{-1}\xx_{j_r}^*]_{\mathcal{B}}
\end{equation} 
are disjoint, then this corollary shows that we can recover all the vectors $\pphi_{j_1}, \, \pphi_{j_2}, \, \ldots, \, \pphi_{j_r}$ from $\KK (\pphi_{j_1}+\pphi_{j_2}+ \cdots +\pphi_{j_1r})$. This indicates that  Method I in many cases should be able to identify several sources. Nevertheless, the content of the vectors \eqref{A8} depends on the projection $\PP$, see \eqref{A7} and \eqref{A3.02}, and the properties of this projection is problem dependent. Below we will explore this issue in more detail for our model problem \eqref{eq1}-\eqref{eq2}. 

\subsubsection{Composite local sources}
\label{composite_sources}
So far we have only studied local sources consisting of a single basis function. Let us now consider a local source which is a sum of several basis functions, i.e., 
\begin{equation}
\label{B2.7}
    f = a_1 \pphi_{j_1}+ a_2 \pphi_{j_2} + \cdots + a_r \pphi_{j_r}, 
\end{equation}
where $a_1, \, a_2, \ldots, a_r$ are constants. Can we roughly recover such a source from its image $\KK f$? 

As in the analysis leading to Corollary \ref{corollary:several_source}, we find that the minimum norm least squares solution $\xx^*$ of 
\begin{equation*}
    \KK \xx = \KK (a_1 \pphi_{j_1}+ a_2 \pphi_{j_2} + a_3 \cdots + a_r \pphi_{j_r})
\end{equation*}
is such that  
\begin{equation}
\label{B2.71}
    \WW^{-1} \xx^* = a_1 \WW^{-1}\xx_{j_1}^* 
    + a_2 \WW^{-1}\xx_{j_2}^* 
    +\cdots  
    + a_r \WW^{-1}\xx_{j_r}^*. 
\end{equation}
Consequently, if $\pphi_{j_1}, \, \pphi_{j_2}, \ldots, \, \pphi_{j_r}$ are basis functions with neighboring local supports, then $a_1 \WW^{-1}\xx_{j_1}^*, \, a_2 \WW^{-1}\xx_{j_2}^*, \ldots, \, a_r \WW^{-1}\xx_{j_r}^*$ will all achieve their maximums (or minimums) in these neighboring supports. We thus expect that $\WW^{-1} \xx^*$ roughly will recover the composite local source \eqref{B2.7}. 

If the right-hand-side $\bb$ in \eqref{A1} does not belong to the range of $\KK$, then the analysis of the potential recovery of a local source becomes even more involved. Typically, one would consider the problem 
\begin{equation*}
    \KK \xx = \hat{\bb}, 
\end{equation*}
where $\hat{\bb}$ represents the orthogonal projection of $\bb$ onto the range of $\KK$. Assuming that there exists a composite function $f$ in the form \eqref{B2.7} such that 
\begin{equation}
    \label{B2.72}
    \KK f = \hat{\bb}, 
\end{equation}
the discussion above suggests that $\WW^{-1} \xx^*$ can yield an approximation of $f$. Here, $\xx^*$ is the minimum norm least squares solution of 
$$
\KK \xx = \hat{\bb} (= \KK (a_1 \pphi_{j_1}+ a_2 \pphi_{j_2} + a_3 \cdots + a_r \pphi_{j_r})).
$$
Whether $\WW^{-1} \xx^*$ also yields an approximation of the true local source, and not only the function $f$ satisfying \eqref{B2.72}, will definitely depend on how "close" $\bb$ is to the range of $\KK$ and the ill-posed nature of \eqref{A1}.

In the numerical experiments section below we primarily study cases where the right-hand-side $\bb$ in \eqref{A1} does not belong to the range of $\KK$: The synthetic observation data $d$ in \eqref{eq1} was generated by solving the forward problem on a finer grid than was used in the inversion process.  

\subsection{Results for elliptic source problems}
\label{elliptic_source_problems}
We will now study the PDE-constrained optimization problem \eqref{eq1}-\eqref{eq2}. Let us discretize the unknown source $f$ in terms of the basis functions 
\begin{align}
\nonumber
\pphi_i &= \frac{1}{\| \mathcal{X}_{\Omega_i}  \|_{L^2(\Omega)}} \mathcal{X}_{\Omega_i} \\
\label{B2.8}
&= A^{-1/2} \, \mathcal{X}_{\Omega_i}, \quad i = 1, 2, \ldots, n, 
\end{align}
where $\Omega_1, \, \Omega_2, \ldots, \Omega_n$ are uniformly sized disjoint grid cells, $\mathcal{X}_{\Omega_i}$ denotes the characteristic function of $\Omega_i$ and $A=|\Omega_1|=|\Omega_2|=\ldots=|\Omega_n|$. That is, the space $X$ associated with \eqref{A1} is 
\[
X = F_h = \mathrm{span} \{ \pphi_1, \, \pphi_2, \ldots, \, \pphi_n \}
\]
and thus consists of piecewise constant functions. (In appendix \ref{Poisson} we explain how such basis functions also can be employed when $\epsilon=0$, i.e., when the PDE in \eqref{eq2} is Poisson's equation.) Throughout this paper we assume that $\Omega$ and the subdomains $\Omega_1, \, \Omega_2, \ldots, \Omega_n$ are such that \eqref{A2} holds.  

Note that the basis functions \eqref{B2.8} are $L^2$-orthonormal and has local support. From the latter property, and the fact that $\PP:X \rightarrow \nullspace{\KK}^{\perp}$ is an orthogonal projection, it follows that we can write \eqref{A3.2} in the form 
\begin{align}
\nonumber
    \WW^{-1} \xx_j^* &= \sum_{i=1}^{n} \frac{1}{\| \PP \pphi_i \|_{L^2(\Omega)}} \left(\PP \pphi_j, \PP \pphi_i \right)_{L^2(\Omega)}  \pphi_i \\
    \nonumber
    &= \sum_{i=1}^{n} \frac{1}{\| \PP \pphi_i \|_{L^2(\Omega)}} \left(\PP \pphi_j, \pphi_i \right)_{L^2(\Omega)}  \pphi_i \\
    \nonumber
    &= A^{-1/2} \sum_{i=1}^{n} \frac{1}{\| \PP \pphi_i \|_{L^2(\Omega)}} \int_{\Omega_i} \PP \pphi_j \, dx \, \,  \pphi_i \\
    \nonumber
    &=A^{-1/2} \frac{1}{\| \PP \pphi_j \|_{L^2(\Omega)}} \int_{\Omega_j} \PP \pphi_j \, dx \, \,  \pphi_j \\
    \nonumber
    & \hspace{0.4cm} + 
    A^{-1/2} \sum_{i=1, i \neq j}^{n} \frac{1}{\| \PP \pphi_i \|_{L^2(\Omega)}} \int_{\Omega_i} \PP \pphi_j \, dx \, \,  \pphi_i. 
\end{align}
Recall that 
$$
\PP \pphi_j \in \nullspace{\KK}^{\perp} \subset X = \mathrm{span}\{\pphi_1, \pphi_2, \ldots, \pphi_n\},
$$ 
and that the functions in $X$ are piecewise constant. Consequently, 
$$
\int_{\Omega_i} \PP \pphi_j \, dx = A \, \PP \pphi_j(z_i) \mbox{ for any } z_i \in \Omega_i,
$$
and we conclude that 
\begin{align}
\nonumber
    \WW^{-1} \xx_j^* &=A^{1/2} \frac{\PP \pphi_j (z_j)}{\| \PP \pphi_j \|_{L^2(\Omega)}} \, \,  \pphi_j \\
    \label{B3}
    & \hspace{0.4cm} + 
    A^{1/2} \sum_{i=1, i \neq j}^{n} \frac{\PP \pphi_j(z_i)}{\| \PP \pphi_i \|_{L^2(\Omega)}} \, \,  \pphi_i,
\end{align}
where $z_1 \in \Omega_1$, $z_2 \in \Omega_2$, $\ldots$, $z_n \in \Omega_n$ are arbitrary points in these subdomains. 

Alternatively, we can express $\WW^{-1} \xx_j^*$ in terms of the projection $\PPN$ onto the nullspace $\nullspace{\KK}$, see \eqref{A4.2}. More specifically, from \eqref{B3} and the orthogonal decomposition $\pphi_j = \PP \pphi_j + \PPN \pphi_j$ 
it follows that
\begin{align}
    \nonumber
    \WW^{-1} \xx_j^* &= A^{1/2} \frac{A^{-1/2} - \PPN \pphi_j(z_j)}{\sqrt{1-\| \PPN \pphi_j \|_{L^2(\Omega)}^2}}  \, \,  \pphi_j \\
    \label{B4}
    & \hspace{0.4cm} - 
    A^{1/2} \sum_{i=1, i \neq j}^{n} \frac{\PPN \pphi_j(z_i)}{\sqrt{1-\| \PPN \pphi_i \|_{L^2(\Omega)}^2}}  \, \,  \pphi_i.  
\end{align}
Here we have used the facts that $\pphi_j(z_j) = A^{-1/2}$ and that $\pphi_j(z_i) = 0$ for $i \neq j$, see \eqref{B2.8}. 

From Theorem \ref{theorem:uniqueness}, \eqref{B3} and \eqref{B4} it follows that 
\begin{align*}
    \frac{\PP \pphi_j (z_j)}{\| \PP \pphi_j \|_{L^2(\Omega)}} &> \frac{\PP \pphi_j(z_i)}{\| \PP \pphi_i \|_{L^2(\Omega)}} \quad \mbox{for } i \neq j, \\
    \frac{A^{-1/2} - \PPN \pphi_j(z_j)}{\sqrt{1-\| \PPN \pphi_j \|_{L^2(\Omega)}^2}} &> 
    \frac{-\PPN \pphi_j(z_i)}{\sqrt{1-\| \PPN \pphi_i \|_{L^2(\Omega)}^2}} \quad \mbox{for } i \neq j, 
\end{align*}
which show that the dominance of the $j$'th component of the Euclidean vector $[\WW^{-1} \xx_j^*]_{\mathcal{B}}$ is determined by the projections $\PP \pphi_j$ or $\PPN \pphi_j$ of $\pphi_j$ onto $\nullspace{\KK}^{\perp}$ and $\nullspace{\KK}$, respectively. As discussed in connection  with \eqref{A3.2} and Theorem \ref{theorem:uniqueness}, we can recover $\pphi_j$ from its image $\KK \pphi_j$ by identifying the largest component of $[\WW^{-1} \xx_j^*]_{\mathcal{B}}$. 
Furthermore, the present analysis reveals that to what degree $\WW^{-1} \xx_j^*$ yields a "smeared out/blurred" approximation of $\pphi_j$ depends on how fast $\PP \pphi_j(z_i)$ (or $\PPN \pphi_j(z_i)$) decays as a function of the distance between $\Omega_i$ and $\Omega_j$. This decay also determines to what extent Method I can identify several local sources, see Corollary \ref{corollary:several_source}.   

\subsubsection{Properties of the projections}\label{subsec:propproj}
Motivated by the investigation presented above, we will know explore the mathematical properties of the orthogonal projections $\PP \pphi_j$ and $\PPN \pphi_j$, see \eqref{A3.02} and \eqref{A4.2}. To this end, consider the forward operator 
$$
\KK: F_h \rightarrow L^2(\partial \Omega), \quad
f \mapsto u|_{\partial \Omega},$$ 
associated with \eqref{eq1}-\eqref{eq2}. Here, $u$ is the solution of the following variational form of the boundary value problem \eqref{eq2}: Determine $u \in H^2(\Omega)$ such that 
\begin{align*}
   \int_{\Omega} (-\Delta u + \epsilon u) \xi  \, dx &= \int_{\Omega} f \xi \, dx 
   \quad \forall \xi \in L^2(\Omega), \\
   \frac{\partial u}{\partial \nn} &= 0  \quad \mbox{on } \partial \Omega.
\end{align*}
This rather non-standard variational form is employed for the sake of simplicity.

If we define 
\begin{equation}
    \label{B5}
    \VNN = \left\{\psi \in H^2(\Omega): \, -\Delta \psi + \epsilon \psi \in X, \, \psi = \frac{\partial\psi}{\partial\mathbf{n}} = 0 \textnormal{ on } \partial\Omega \right\}, 
\end{equation}
then the nullspace of $\KK$ can be characterized as follows 
\begin{equation}
\label{B6}
\nullspace{\KK} =  \left\{q = -\Delta \psi + \epsilon \psi, \, \psi \in \VNN \right\}.
\end{equation}
Any function $r \in \nullspace{\KK}^{\perp}$ satisfies 
$$
\int_{\Omega} r q \, dx = 0 \quad \forall q \in \nullspace{\KK}
$$
or 
\begin{equation}
\label{B0}
\int_{\Omega} r \cdot (-\Delta \psi + \epsilon \psi)  \, dx = 0 \quad \forall \psi \in \VNN. 
\end{equation}
We may, in view of Green's formula/integration by parts, say that $r$ is a discrete {\em very weak solution} of
\begin{equation}
\label{B1}
-\Delta r + \epsilon r = 0. 
\end{equation}
(Any member of $\nullspace{\KK}^{\perp} \subset X$ is piecewise constant, and \eqref{B1} is thus not  meaningful for such functions.) We here use the word {\em discrete} because $r \in \nullspace{\KK}^{\perp} \subset X$ and $X$ is finite dimensional.

Recall that $\xx_j^* = \PP \pphi_j \in \nullspace{\KK}^{\perp}$, see \eqref{A3.01} and \eqref{A3.02}, and we conclude that the minimum norm least squares solution $\xx_j^*$ of \eqref{A3} is the best discrete approximation of $\pphi_j$ satisfying, in a very weak sense, 
\begin{equation}
\label{B1.1}
-\Delta \PP \pphi_j + \epsilon \PP \pphi_j = 0.
\end{equation}
In other words, provided that $\epsilon=0$, $\PP \pphi_j$ is the best discrete very weak harmonic approximation of $\pphi_j$. 

Having characterized $\PP \pphi_j$, we turn our attention toward $\PPN \pphi_j$. 
Since $\PPN \pphi_j$ belongs to $\nullspace{\KK}$, it has a "generating" function $\tau_j \in \VNN$, see \eqref{B5} and \eqref{B6}: 
$$
\PPN \pphi_j = -\Delta \tau_j + \epsilon \tau_j.
$$ 
Choosing $r=\PP \pphi_j = \pphi_j-\PPN \pphi_j$ in \eqref{B0} yields that this "generating" function must satisfy the following discrete weak version of a fourth order PDE: Find $\tau_j \in \VNN$ such that 
\begin{equation}
\label{B2}
   \int_{\Omega} (-\Delta \tau_j + \epsilon \tau_j) \cdot (-\Delta \psi + \epsilon \psi)  \, dx = 
   \int_{\Omega_j} \pphi_j \cdot (-\Delta \psi + \epsilon \psi)  \, dx \quad \forall \psi \in \VNN,  
\end{equation}
where we have invoked the fact that $\pphi_j$ has the local support $\Omega_j$, see \eqref{B2.8}. Note that, with $\epsilon=0$, \eqref{B2} roughly\footnote{The function $\pphi_j$, see the right-hand-side of the PDE in \eqref{B7}, must be sufficiently differentiable and $\mathrm{supp} (\pphi_j) \subset \Omega_j$ in order for \eqref{B2} to be the weak version of \eqref{B7} (when $\epsilon=0$). The basis functions defined in \eqref{B2.8} do not satisfy the necessary regularity conditions because they are discontinuous.} becomes the standard discrete, using a somewhat peculiar discretization space $\VNN$, weak version of the inhomogeneous biharmonic equation with homogeneous boundary conditions: 
\begin{equation}
\label{B7}
\begin{split}
    \Delta^2 \tau_j &= \Delta \pphi_j \quad \mbox{in } \Omega, \\
    \tau_j = \frac{\partial\tau_j}{\partial\mathbf{n}} &= 0 \quad \textnormal{on } \partial\Omega.
\end{split}
\end{equation}

The solution of a second or fourth order elliptic PDE depends significantly on the size and shape of the involved domain $\Omega$. Hence, \eqref{B1.1} and \eqref{B2} show that the aforementioned decaying properties of $\PP \pphi_j$ and $\PPN \pphi_j$ depend on $\Omega$ and the position of the true source relative to the boundary $\partial \Omega$. Hence, the "sharpness" of the reconstruction/recovery $\WW^{-1} \xx^*_j$ of $\pphi_j$, as well as the possibility of identifying several sources with Method I, will depend on the geometrical properties of $\Omega$ -- each domain must be studied separately. This issue is explored in more detail in the numerical experiments section below.   

\section{Methods II and III} 
If we apply weighted Tikhonov regularization to \eqref{A1}, we obtain the regularized solutions 
\begin{equation}
\label{C1}
\zz_{\alpha} = \argmin_{\zz} \left\{ \frac{1}{2} \| \KK \zz - \bb \|_Y^2 + \frac{1}{2} \alpha \| \WW \zz \|_X^2 \right\}, 
\end{equation}
where, in this paper, the regularization operator $\WW$ is defined in \eqref{A3.1}. In \cite{Elv20} we also, in addition to Method I described above, introduced the following two schemes for identifying sources: 
\begin{description}
\item{\bf \underline{Method II}} Defining $\yy = \WW \zz$, we obtain from \eqref{C1},  
\begin{equation}
\label{C2}
\yy_{\alpha} = \argmin_{\yy} \left\{ \frac{1}{2} \| \KK \WW^{-1} \yy - \bb \|_Y^2 + \frac{1}{2} \alpha \| \yy \|_X^2 \right\}, 
\end{equation}
which is Method II. 
Theorem 4.3 in \cite{Elv20} expresses that the minimum norm least squares solution $\yy_j^*$ of 
\begin{equation}
\nonumber 
\KK \WW^{-1} \yy = \KK \pphi_j, 
\end{equation}
satisfies 
\begin{align}
\label{C3}
    \left\| \pphi_j - \frac{\yy_j^*}{\| \PP \pphi_j \|_X} \right\|_X \leq  \| \pphi_j - \WW^{-1} \xx_j^* \|_X,  
\end{align}
where $\xx_j^*$ is the minimum norm least squares solution of \eqref{A3}. Recall that $\yy_j^* = \lim_{\alpha \rightarrow 0 } \yy_{j,\alpha}$, where 
\[
\yy_{j,\alpha} = \argmin_{\yy} \left\{ \frac{1}{2} \| \KK \WW^{-1} \yy - \KK \pphi_j \|_Y^2 + \frac{1}{2} \alpha \| \yy \|_X^2 \right\}. 
\]
Consequently, \eqref{C3} shows that, for small $\alpha > 0$, a scaled version of Method II can yield more accurate recoveries than Method I of the individual basis functions from their images under $\KK$. (Method I is defined in \eqref{A4.1}.)
\item{\bf \underline{Method III}} This method is defined by \eqref{C1}, i.e., the outcome of the scheme is $\zz_{\alpha}$. Note that there is a simple connection between methods II and III:  
\begin{equation}
\label{C4}
\zz_{\alpha} = \WW^{-1} \yy_{\alpha}. 
\end{equation}
Hence, a result similar to \eqref{C3}, though not that strong, also holds for Method III, see \cite{Elv20} for further details. 
\end{description}

\subsection{Several sources} 
Let us briefly comment on Method II's ability to localize several sources. Similar to \eqref{A6.01} we consider 
\[
\KK \WW^{-1} \yy = \KK (\pphi_{j_1}+\pphi_{j_2}+ \cdots +\pphi_{j_r}), 
\]
which has the minimum norm least squares solution 
\begin{align}
    \nonumber
    \yy^* &= (\KK \WW^{-1})^{\dagger} \KK \pphi_{j_1} + (\KK \WW^{-1})^{\dagger} \KK \pphi_{j_2} + \cdots + (\KK \WW^{-1})^{\dagger} \KK \pphi_{j_r} \\
    &= \yy_{j_1}^*+\yy_{j_1}^*+ \cdots + \yy_{j_r}^*, 
\end{align}
where $\yy_{j_s}^*$, for $s=1,2, \ldots,r$, represents the minimum norm least squares solution of 
\begin{equation*}
 \KK \WW^{-1} \yy = \KK \pphi_{j_s}.   
\end{equation*}
Since $\yy_{j_s}^*$, for $s=1,2, \ldots,r$, satisfies an inequality in the form \eqref{C3}, we conclude that $\yy^*$ is a sum of vectors which can yield better recoveries of the individual basis function from their images under $\KK$ than Method I. We therefore expect that Method II can separate several local sources whenever Method I can do it. 

Invoking \eqref{C4} leads to a similar type of argument for Method III's ability to identify two, or more, sources. We omit the details. 

\section{Numerical experiments}
We avoided inverse crimes by generating the synthetic observation data $d$ in \eqref{eq1} using a finer grid for the state $u$ than was employed for the computations of the inverse solutions: $h_{\mathrm{forward}} = 0.5 \cdot h_{\mathrm{inverse}}$, where $h_{\mathrm{forward}}$ and $h_{\mathrm{inverse}}$ are the grid parameters associated with the meshes used to produce $d$ and 
the inverse solutions, respectively. More specifically, on the unit square we employed a $64 \times 64$ mesh for the forward computations and a $32 \times 32$ grid for computing the unknown $u$ by solving \eqref{eq1}-\eqref{eq2}. Except for the results presented in Example 4, a coarser mesh, $16 \times 16$, was applied for the unknown source $f$ in the numerical solution of \eqref{eq1}-\eqref{eq2}. 

The triangulations of the non-square geometries were obtained by "removing" grid cells from the triangulations of the associated square domains. We used the Fenics software to generate the meshes and the matrices, and the optimization problem \eqref{eq1}-\eqref{eq2} was solved with Matlab in terms of the associated optimality system. In all the simulations $\epsilon= 1$, and no noise was added to the observation data $d$, see \eqref{eq1}-\eqref{eq2}. (Some simulations with noisy observation data are presented in \cite{Elv20}.)

\subsection*{Example 1: L-shaped versus square geometry}
Figure \ref{fig:l_square} displays the numerical results obtained by solving \eqref{eq1}-\eqref{eq2} for an L-shaped geometry and a square-shaped geometry, respectively. The true sources, shown in panels (a) and (b), are located at the same positions for both geometries. 

Method I fails to separate the two sources on the square domain, but works well for the L-shaped case. On the other hand, methods II and III handle both geometries adequately, noticing that the separation is more pronounced for the L-formed domain. This is consistent with the mathematical result \eqref{C3}, which expresses that a  scaled version of Method II yields more $L^2$-accurate approximations than Method I.  
These results illuminate the impact of the geometry on the inverse source problem and suggest that convex domains lead to harder identification tasks than non-convex regions. 

In these simulations the two true local sources did not equal a single basis function $\phi_j$, but was instead defined as a sum of four basis functions with neighbouring supports. Hence, for each of the two local sources, the considerations presented in subsection \ref{composite_sources} are relevant. 
\begin{figure}[h]
    \centering
    \begin{subfigure}[b]{0.45\linewidth}        
        \centering
        \includegraphics[width=\linewidth]{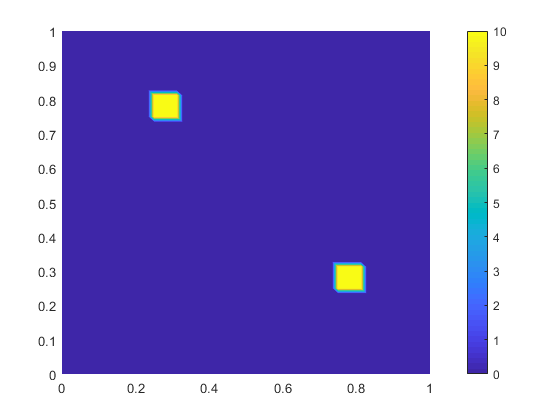}
        \caption{True source}
    \end{subfigure}
    \begin{subfigure}[b]{0.45\linewidth}        
        \centering
        \includegraphics[width=\linewidth]{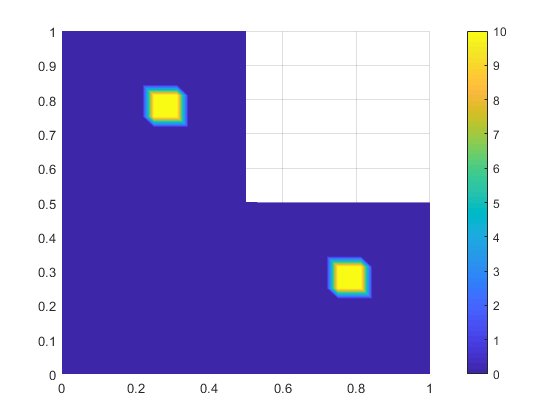}
        \caption{True source}
    \end{subfigure}\par
    \begin{subfigure}[b]{0.45\linewidth}        
        \centering
        \includegraphics[width=\linewidth]{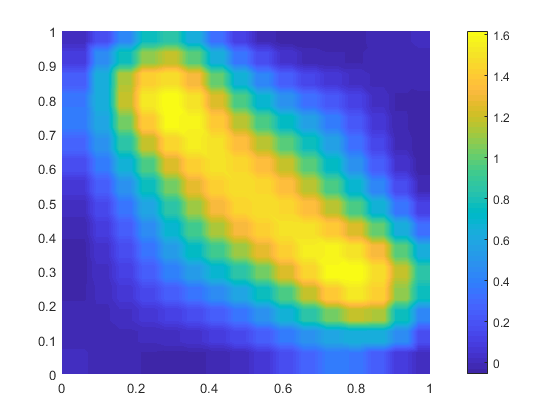}
        \caption{Method I}
    \end{subfigure}
    \begin{subfigure}[b]{0.45\linewidth}        
    \centering
    \includegraphics[width=\linewidth]{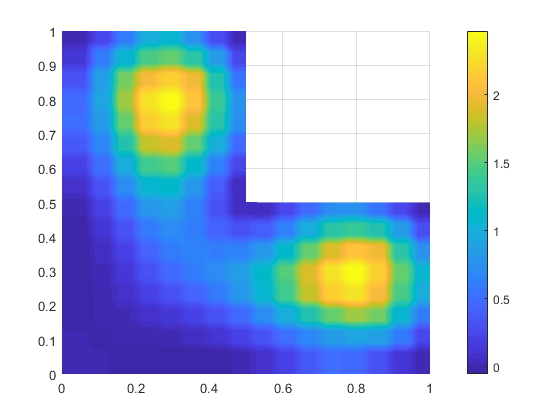}
    \caption{Method I}
    \end{subfigure}\par
    \begin{subfigure}[b]{0.45\linewidth}        
        \centering
        \includegraphics[width=\linewidth]{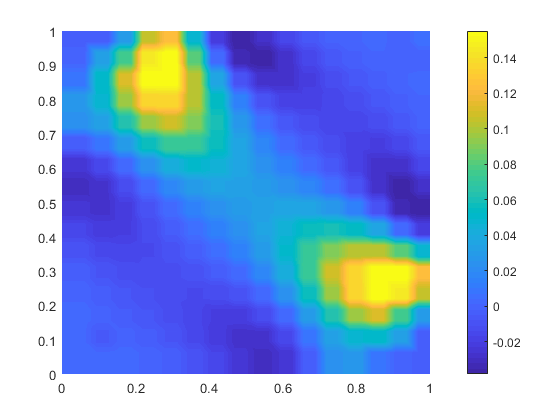}
        \caption{Method II}
    \end{subfigure}
    \begin{subfigure}[b]{0.45\linewidth}        
    \centering
    \includegraphics[width=\linewidth]{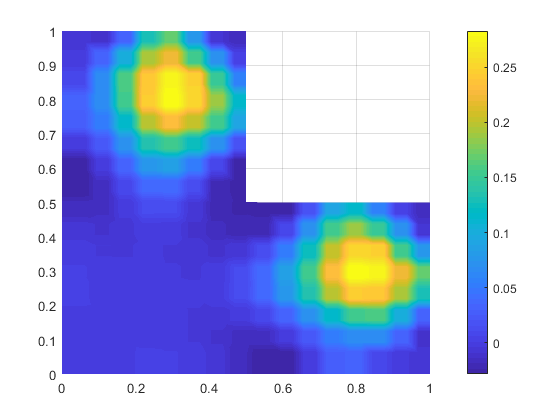}
    \caption{Method II}
    \end{subfigure}\par
    \begin{subfigure}[b]{0.45\linewidth}        
        \centering
        \includegraphics[width=\linewidth]{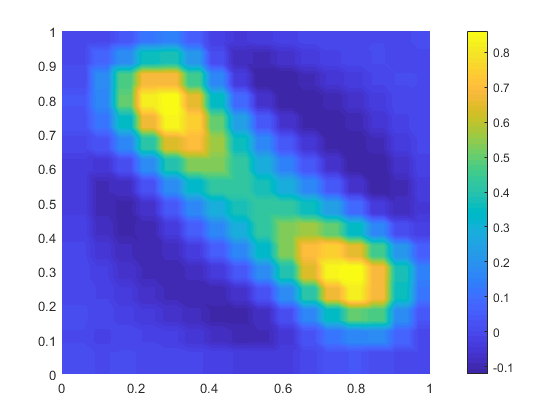}
        \caption{Method III}
    \end{subfigure}
    \begin{subfigure}[b]{0.45\linewidth}        
    \centering
    \includegraphics[width=\linewidth]{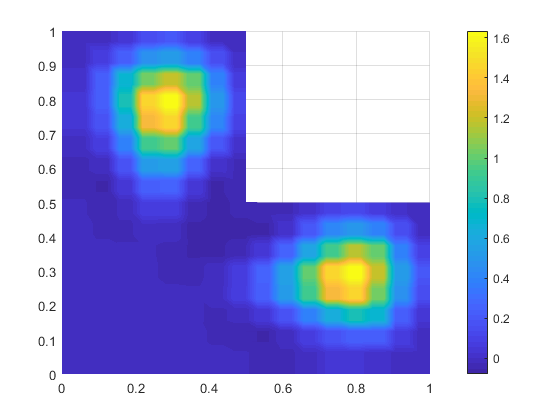}
    \caption{Method III}
    \end{subfigure}
    \caption{Example 1. Comparison of the true sources and the inverse solutions for an L-shaped and square-shaped geometry. The regularization parameter was $\alpha = 10^{-6}$.}
    \label{fig:l_square}
\end{figure}

\subsection*{Example 2: Square versus horseshoe} 
Figure \ref{fig:horse_square} shows computations performed with a horseshoe-shaped domain and a square region. 
In these simulations each of the two true sources consisted of a single basis function. Hence, Corollary \ref{corollary:several_source} is directly applicable. Again we observe that the source identification works better for a non-convex domain than for a convex region. 

We also performed computations with partial boundary observations $d$, see Figure \ref{fig:partial}: boundary observation data was only available for the part of the boundary marked with red in panel (a). We observe that this reduces the quality of the reconstruction of the true sources, compare figures \ref{fig:horse_square}  and \ref{fig:partial}, and that Method I works somewhat better than methods II and III in this case. 
\begin{figure}[h]
    \centering
    \begin{subfigure}[b]{0.45\linewidth}        
        \centering
        \includegraphics[width=\linewidth]{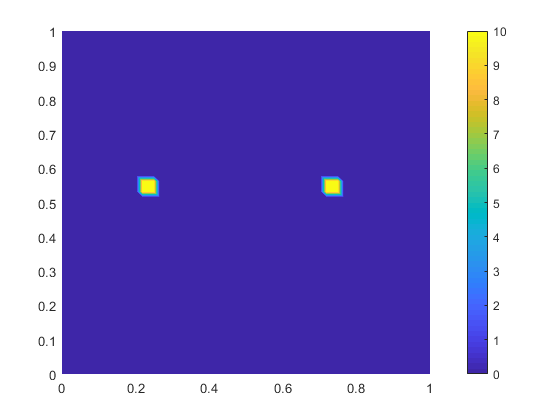}
        \caption{True source}
    \end{subfigure}
    \begin{subfigure}[b]{0.45\linewidth}        
        \centering
        \includegraphics[width=\linewidth]{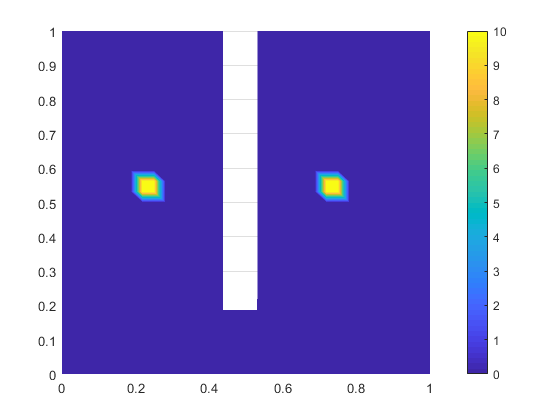}
        \caption{True source}
    \end{subfigure}\par
    \begin{subfigure}[b]{0.45\linewidth}        
        \centering
        \includegraphics[width=\linewidth]{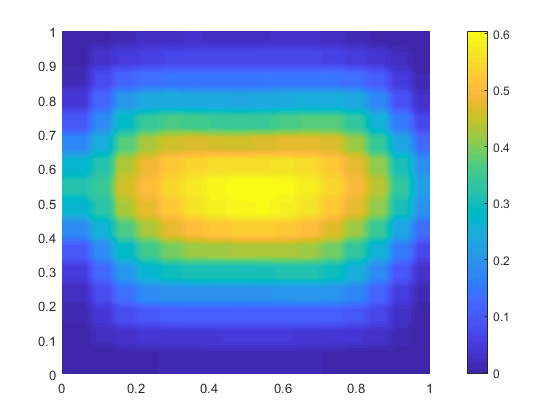}
        \caption{Method I}
    \end{subfigure}
    \begin{subfigure}[b]{0.45\linewidth}        
    \centering
    \includegraphics[width=\linewidth]{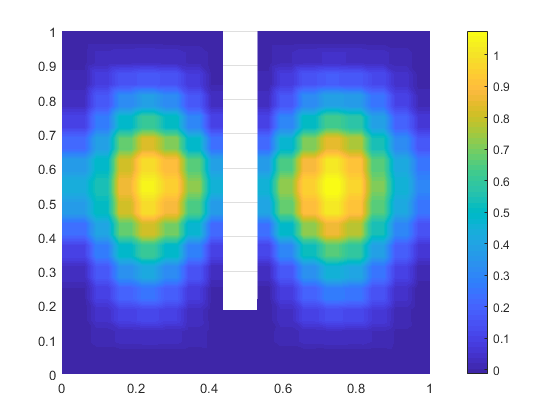}
    \caption{Method I}
    \end{subfigure}\par
    \begin{subfigure}[b]{0.45\linewidth}        
        \centering
        \includegraphics[width=\linewidth]{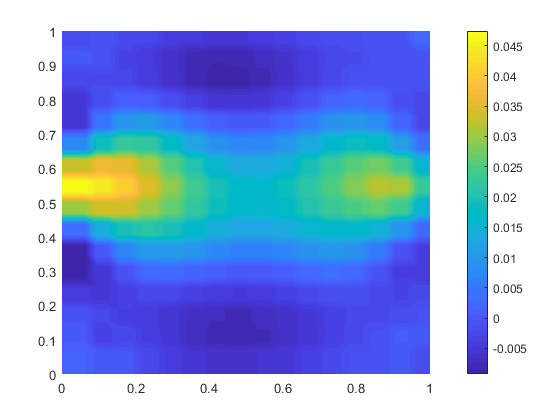}
        \caption{Method II}
    \end{subfigure}
    \begin{subfigure}[b]{0.45\linewidth}        
    \centering
    \includegraphics[width=\linewidth]{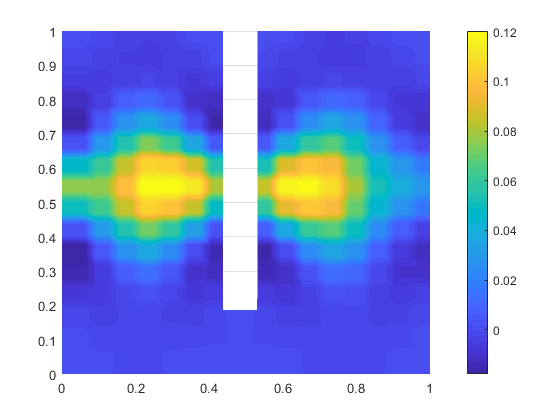}
    \caption{Method II}
    \end{subfigure}\par
    \begin{subfigure}[b]{0.45\linewidth}        
        \centering
        \includegraphics[width=\linewidth]{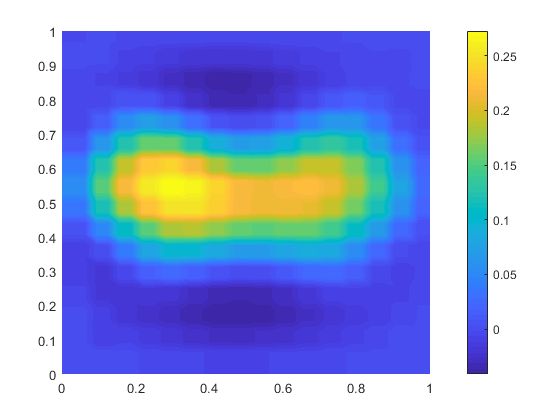}
        \caption{Method III}
    \end{subfigure}
    \begin{subfigure}[b]{0.45\linewidth}        
    \centering
    \includegraphics[width=\linewidth]{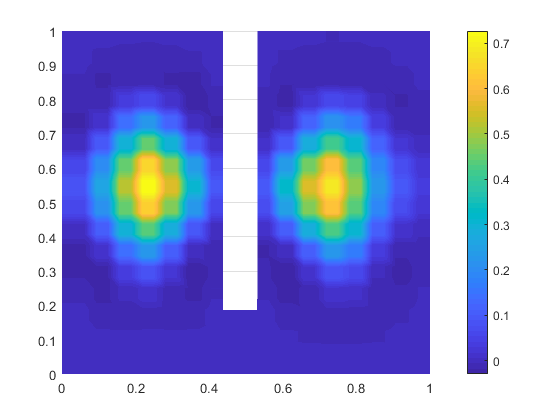}
    \caption{Method III}
    \end{subfigure}
    \caption{Example 2. Comparison of the true sources and the inverse solutions for a horseshoe-shaped and a square-shaped geometry. The regularization parameter was $\alpha = 10^{-6}$.}
    \label{fig:horse_square}
\end{figure}
\begin{figure}[h]
    \centering
    \begin{subfigure}[b]{0.45\linewidth}        
        \centering
        \includegraphics[width=\linewidth]{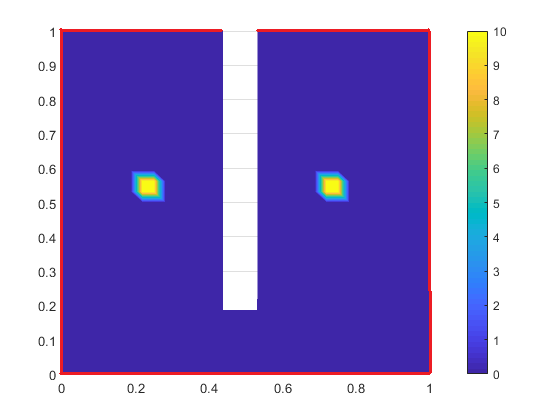}
        \caption{True source.}
    \end{subfigure}
    \begin{subfigure}[b]{0.45\linewidth}        
        \centering
        \includegraphics[width=\linewidth]{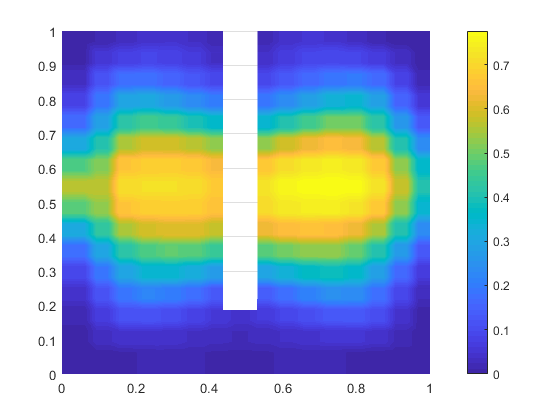}
        \caption{Method I.}
    \end{subfigure}\par
    \begin{subfigure}[b]{0.45\linewidth}        
        \centering
        \includegraphics[width=\linewidth]{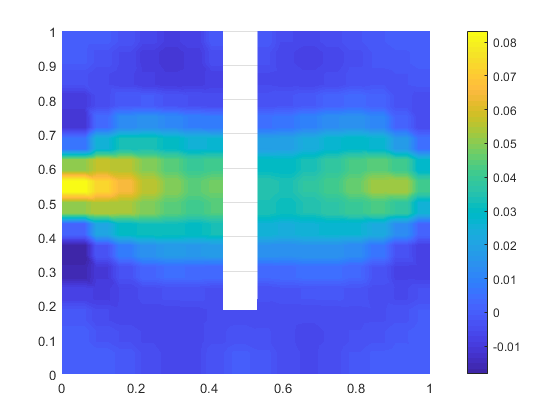}
        \caption{Method II.}
    \end{subfigure}                                                                                                                                                                                                                      
    \begin{subfigure}[b]{0.45\linewidth}        
        \centering
        \includegraphics[width=\linewidth]{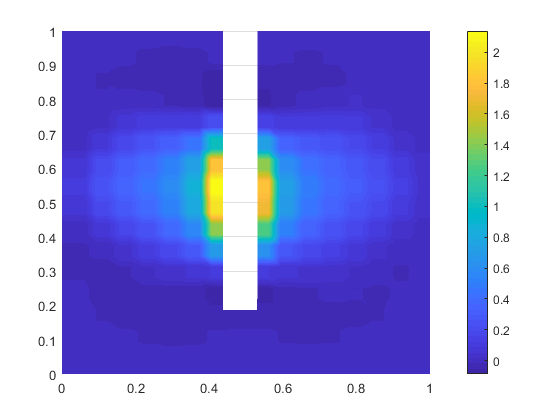}
        \caption{Method III.}
    \end{subfigure}\par
    \caption{Example 2 with partial boundary observation observations: $d$ in \eqref{eq1} is only defined along the red line segments in panel (a) (and the boundary integral in \eqref{eq1} is adjusted accordingly). Comparison of the true sources and the inverse solutions. The regularization parameter was $\alpha = 10^{-6}$.}
    \label{fig:partial}
\end{figure}

\subsection*{Example 3: Rectangles, distance to the boundary}
So far we have compared convex and non-convex regions. We now illuminate how the distance from the source(s) to the boundary of the domain influence the quality of the recovery, see Figure \ref{fig:rectangles}: The identification of the three sources improves as the distance to the boundary decreases. Also, methods II and III yield better results than Method I.  

In this example each of the true local sources are composed of several basis functions with neighbouring supports, cf. subsection \ref{composite_sources} for further details.  
\begin{figure}[h]
    \centering
    \begin{subfigure}[b]{0.32\linewidth}        
        \centering
        \includegraphics[width=\linewidth]{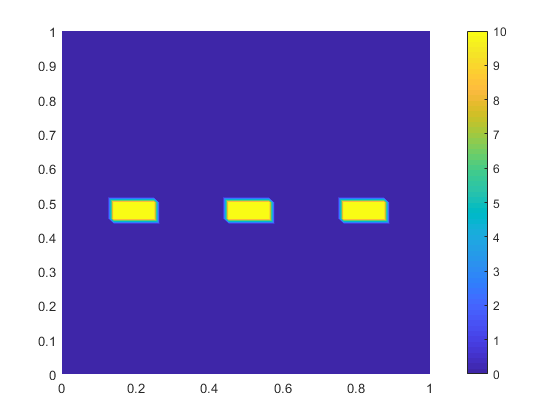}
        \caption{True source.}
    \end{subfigure}
    \begin{subfigure}[b]{0.32\linewidth}        
        \centering
        \includegraphics[width=\linewidth]{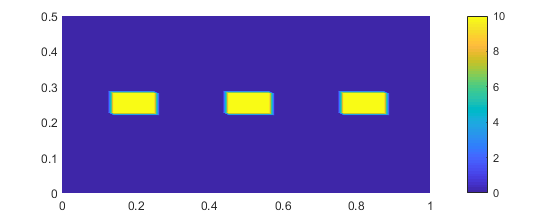}
        \caption{True source.}
    \end{subfigure}
    \begin{subfigure}[b]{0.32\linewidth}        
        \centering
        \includegraphics[width=\linewidth]{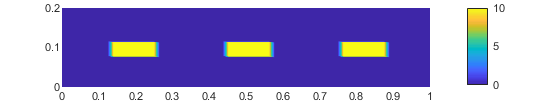}
        \caption{True source.}
    \end{subfigure}\par
    \begin{subfigure}[b]{0.32\linewidth}        
        \centering
        \includegraphics[width=\linewidth]{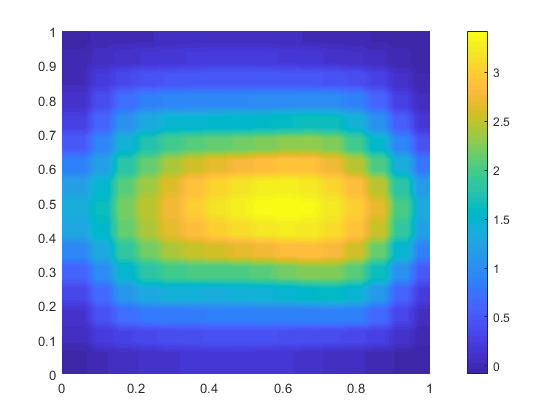}
        \caption{Method I.}
    \end{subfigure}
    \begin{subfigure}[b]{0.32\linewidth}        
        \centering
        \includegraphics[width=\linewidth]{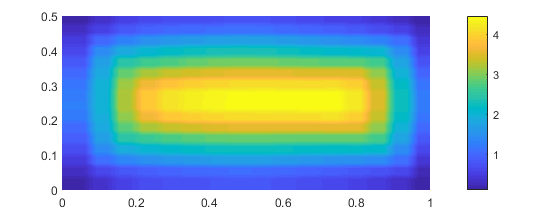}
        \caption{Method I.}
    \end{subfigure}
    \begin{subfigure}[b]{0.32\linewidth}        
        \centering
        \includegraphics[width=\linewidth]{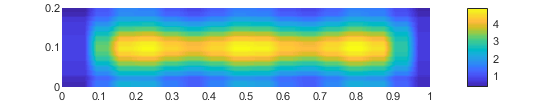}
        \caption{Method I.}
    \end{subfigure}\par   
    \begin{subfigure}[b]{0.32\linewidth}        
        \centering
        \includegraphics[width=\linewidth]{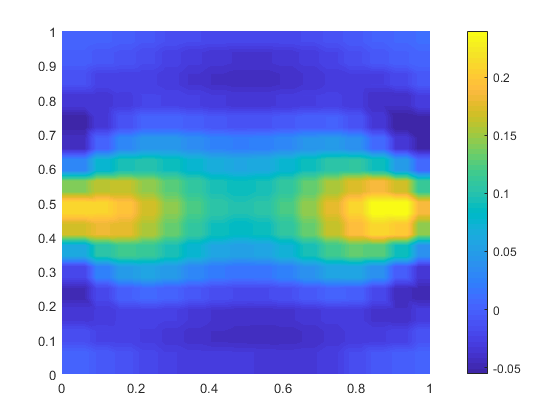}
        \caption{Method II.}
    \end{subfigure}
    \begin{subfigure}[b]{0.32\linewidth}        
        \centering
        \includegraphics[width=\linewidth]{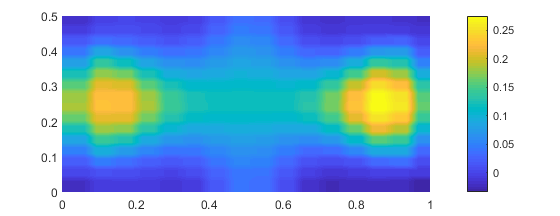}
        \caption{Method II.}
    \end{subfigure}
    \begin{subfigure}[b]{0.32\linewidth}        
        \centering
        \includegraphics[width=\linewidth]{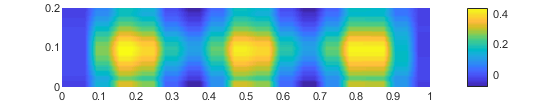}
        \caption{Method II.}
    \end{subfigure}\par   
    \begin{subfigure}[b]{0.32\linewidth}        
        \centering
        \includegraphics[width=\linewidth]{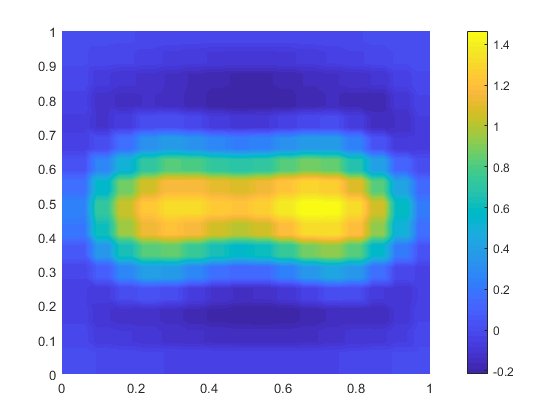}
        \caption{Method III.}
    \end{subfigure}
    \begin{subfigure}[b]{0.32\linewidth}        
        \centering
        \includegraphics[width=\linewidth]{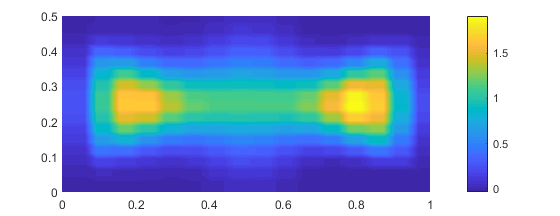}
        \caption{Method III.}
    \end{subfigure}
    \begin{subfigure}[b]{0.32\linewidth}        
        \centering
        \includegraphics[width=\linewidth]{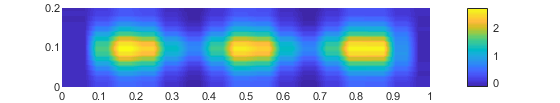}
        \caption{Method III.}
    \end{subfigure}\par       
    \caption{Example 3. Comparison of the true sources and the inverse solutions for three different rectangles. The regularization parameter was $\alpha = 10^{-4}$ and $\Omega = (0,1) \times (0,\gamma)$, where $\gamma = 1, 0.5$ and $0.2$ for the left, middle and right panels, respectively.}
    \label{fig:rectangles}
\end{figure}

\subsection*{Example 4: A smooth local source}
In examples 1-3 we considered true sources which are piecewise constant. Figure \ref{fig:composite} shows results obtained with the true smooth source
\begin{equation}\nonumber
    f = e^{-10(x_1-0.3)^2 - 5(x_2-0.25)^2}.
\end{equation}
Methods I and III handle this case very well, but the outcome of Method II is not very good. The outcome of Method I is as one could anticipate from the discussion presented in subsection \ref{composite_sources}, but we do not have a good understanding of the rather poor performance of Method II for this particular problem. 
\begin{figure}[h]
    \centering
    \begin{subfigure}[b]{0.45\linewidth}        
        \centering
        \includegraphics[width=\linewidth]{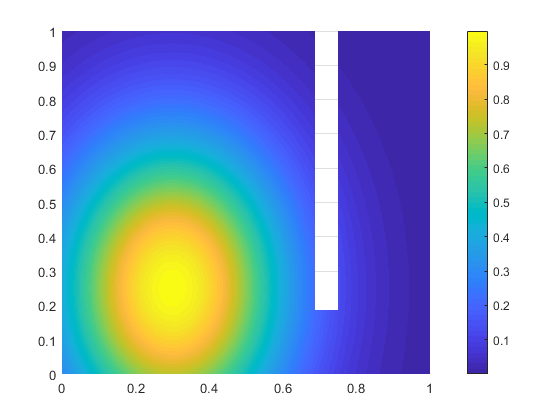}
        \caption{True source.}
    \end{subfigure}
    \begin{subfigure}[b]{0.45\linewidth}        
        \centering
        \includegraphics[width=\linewidth]{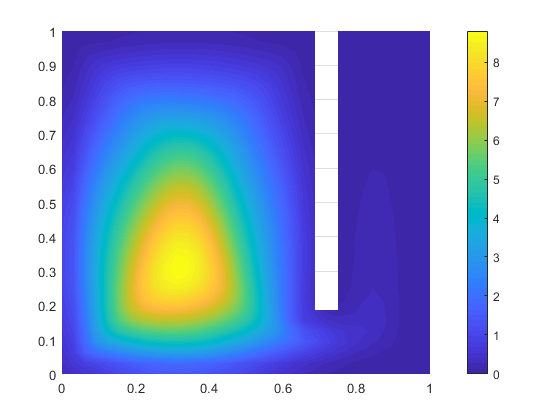}
        \caption{Method I.}
    \end{subfigure}\par
    \begin{subfigure}[b]{0.45\linewidth}        
        \centering
        \includegraphics[width=\linewidth]{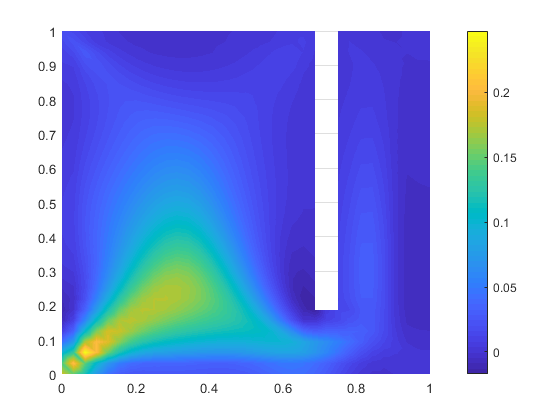}
        \caption{Method II.}
    \end{subfigure}
    \begin{subfigure}[b]{0.45\linewidth}        
        \centering
        \includegraphics[width=\linewidth]{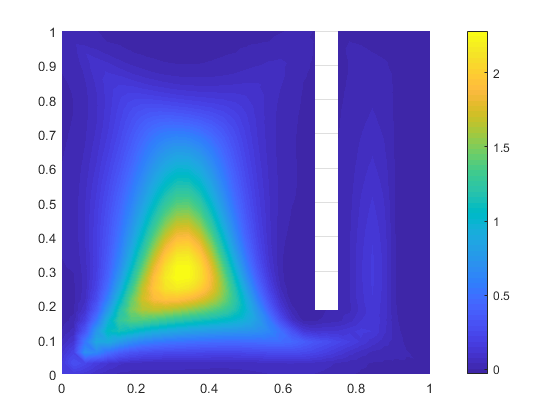}
        \caption{Method III.}
    \end{subfigure}\par
    \caption{Example 4. Comparison of the true smooth source and the inverse solutions. The regularization parameter was $\alpha = 10^{-6}$.}
    \label{fig:composite}
\end{figure}

\subsection*{Example 5: Identifying local constant sources with a known magnitude}\label{subsec:constsource}
%
If the magnitude of the local sources is known, we only need to recover the size and positions of the sources. We will now briefly explain how Method I can be used to handle such cases. Recall Corollary \ref{corollary:several_source}, which expresses that Method I in many cases can detect the index, and thereby the position, of the individual local sources. This leads to the following three-stage optimization procedure: 
\begin{enumerate}
    \item Apply Method I, i.e., compute $\WW^{-1} \xx_{\alpha}$, where $\xx_{\alpha}$ is the outcome of employing standard Tikhonov  regularization \eqref{A0}, and $\WW$ is defined in \eqref{A3.1}. 
    \item Retrieve the positions $p_1, p_2, ..., p_m$ of all the local maximums of $\WW^{-1} \xx_{\alpha}$. 
    \item Use $p_1, p_2, ..., p_m$ as centers of simple geometrical objects, e.g., circles, and solve the optimization problem 
        \begin{equation*} 
            \min_{r_1, r_2 ..., r_m \in \mathbb{R}_+} \frac{1}{2}\|u-d\|^2_{L^2(\partial\Omega)}
        \end{equation*}
        subject to 
        \begin{equation*}
            \begin{split}
                -\Delta u + \epsilon u &= c \sum_{i=1}^m \mathcal{X}_{B_{r_i}(p_i)}  \quad \mbox{in } \Omega, \\
                \frac{\partial u}{\partial \nn} &= 0  \quad \mbox{on } \partial \Omega, 
            \end{split}
        \end{equation*}
        where $B_{r_i}(p_i) = \{x\in \Omega: \|x - p_i\| < r_i\}$, and $c$ is the known magnitude of the source(s).
\end{enumerate}

Panel (c) in Figure \ref{fig:radii} shows that this procedure can work very well: Even though Method I almost fails to detect the small local source in the lower right corner of the L-shaped domain, see panel (b), the radii optimization approach handles the case very well. 
\begin{figure}[h]
    \centering
    \begin{subfigure}[b]{0.5\linewidth}        
        \centering
        \includegraphics[width=\linewidth]{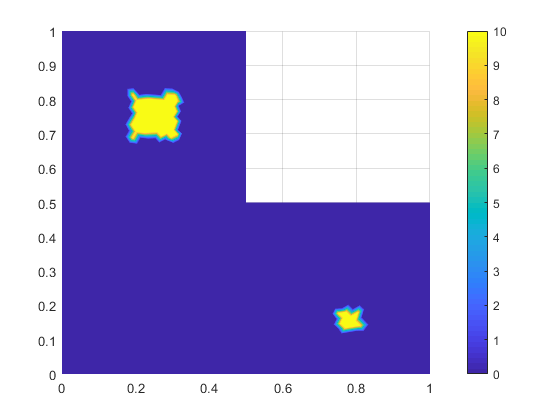}
        \caption{True source.}
    \end{subfigure}\par
    \begin{subfigure}[b]{0.5\linewidth}        
        \centering
        \includegraphics[width=\linewidth]{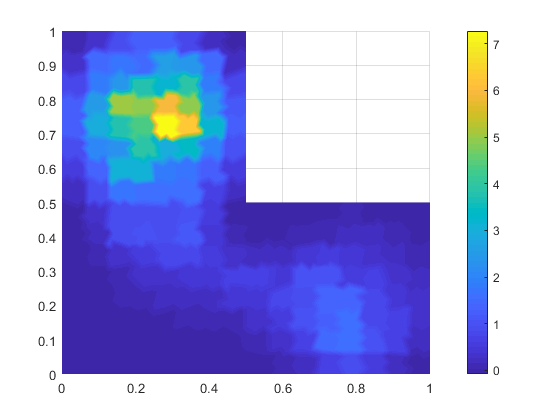}
        \caption{Method I.}
    \end{subfigure}\par
    \begin{subfigure}[b]{0.5\linewidth}        
        \centering
        \includegraphics[width=\linewidth]{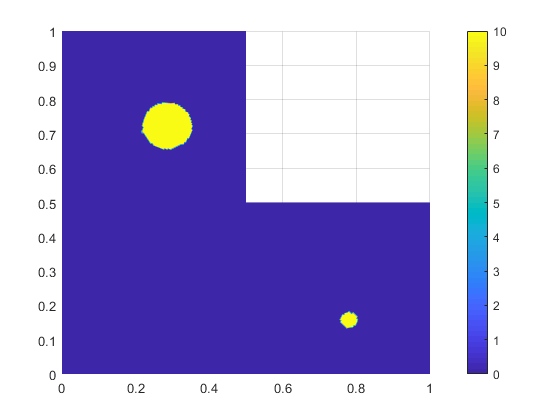}
        \caption{Radii optimization.}
    \end{subfigure}
    \caption{Results obtained with the three-stage algorithm described in Example 5. Panel (a) depicts the true sources, and panel (b) shows the inverse solution computed with Method I, where the regularization parameter was $\alpha = 10^{-6}$. Finally, panel (c) displays the outcome of the radii optimization algorithm.}
    \label{fig:radii}
\end{figure}

\subsection*{Discussion} 
In some cases methods II and/or III work better than Method I, see figures \ref{fig:l_square}, \ref{fig:horse_square} and \ref{fig:rectangles}. This is in contrast to the results presented in figures \ref{fig:partial} and \ref{fig:composite} for which Method I provides the best source identification. Hence, we can not advice that only one of the algorithms should be used. Method I should be applied to get a rough picture of the location of the sources, since this scheme can localize the position of the maximum of single sources. The outputs of methods II and III may yield less "smeared out/blurred" images of the true sources, but should only be trusted if their localization is consistent with the results obtained with Method I. 

Our mathematical analysis shows that the ability to identify internal local sources from Dirichlet boundary data highly depends on the geometry of the domain and the position of the true sources relative to the boundary of this domain, see the analysis  leading to \eqref{B1.1} and \eqref{B2}. The numerical experiments exemplify this, and, in particular, source identification for non-convex domains can lead to better recovery than computations performed with convex domains of approximately the same size.  

\appendix
\section{Poisson's equation} 
\label{Poisson}
In many applications $\epsilon = 0$, and the PDE in \eqref{eq2} becomes Poisson's equation. Then the boundary value problem \eqref{eq2}, for a given $f$, does not have a unique solution, and $f$ must satisfy the complementary condition 
\begin{equation}
    \label{App1.1}
    \int_{\Omega} f \, dx = 0. 
\end{equation}
Note that the basis functions \eqref{B2.8} do not satisfy this condition. In fact, it may be difficult to construct convenient $L^2$-orthonormal basis functions with local supports which obey \eqref{App1.1}. To handle this matter, one can "replace" the right-hand-side $f$ in the state equation with $f - | \Omega |^{-1} \int_{\Omega} f \, dx$:
    \begin{equation} \label{App1.2}
        \min_{(f,u) \in F_h \times H^1(\Omega)} \left\{ \frac{1}{2}\|u-d\|_{L^2(\partial\Omega)}^2 + \frac{1}{2}\alpha\|\WW f\|_{L^2(\Omega)}^2 \right\}
    \end{equation}
    subject to 
    \begin{equation} \label{App1.3}
    \begin{split}
        -\Delta u &= f - | \Omega |^{-1} \int_{\Omega} f \, dx \quad \mbox{in } \Omega, \\
        \frac{\partial u}{\partial \nn} &= 0  \quad \mbox{on } \partial \Omega. 
    \end{split}
    \end{equation}
Note that \eqref{App1.3} is meaningful for any $f \in L^2(\Omega)$, and it follows that we can use basis functions in the form \eqref{B2.8} to discretize the control.  

Let us also make a few remarks about the forward operator associated with \eqref{App1.2}-\eqref{App1.3}. 
Assume that $(f^*, u^*)$ solves \eqref{App1.2}-\eqref{App1.3}.
Note that, if $u^*$ solves $\eqref{App1.3}$, so does $u^*+c$ for any constant $c$. Consider the function 
\begin{equation*}
    g(c)=\frac{1}{2}\|u^*+c-d\|_{L^2(\partial\Omega)}^2 + \frac{1}{2}\alpha\|\WW f\|_{L^2(\Omega)}^2, \, c \in \mathbb{R}.
\end{equation*}
The optimality condition 
\begin{equation*}
    g'(0) = 0
\end{equation*}
yields that 
\begin{equation*}
    \int_{\partial\Omega} u^* \, dS = \int_{\partial\Omega} d \, dS = 0, 
\end{equation*}
provided that the data $d$, which typically is a measured potential, has been normalized such that 
\begin{equation*}
    \int_{\partial\Omega} d \, dS = 0. 
\end{equation*}
Consequently, the forward operator associated with \eqref{App1.2}-\eqref{App1.3} is 
$$
\KK: F_h \rightarrow L^2(\partial \Omega), \quad
f \mapsto u|_{\partial \Omega},
$$
where $u$, for a given $f$, denotes the solution of the boundary value problem \eqref{App1.3} which satisfies 
\begin{equation*}
     \int_{\partial\Omega} u \, dS = 0.
\end{equation*}
We also note that, in this case any constant function $f(x)=C$, for all $x \in \Omega$, belongs to the nullspace of $\KK$. Consequently, the minimum norm least squares solution of $\KK f = d$ will have zero integral.

\bibliographystyle{abbrv}
\bibliography{references}

\begin{thebibliography}{10}

\bibitem{Abdel15}
B.~Abdelaziz, A.~{El Badia}, and A.~{El Hajj}.
\newblock Direct algorithms for solving some inverse source problems in 2{D}
  elliptic equations.
\newblock {\em Inverse Problems}, 31(10):105002, 2015.

\bibitem{babda09}
A.~Ben~Abda, F.~Ben~Hassen, J.~Leblond, and M.~Mahjoub.
\newblock Sources recovery from boundary data: A model related to
  electroencephalography.
\newblock {\em Mathematical and Computer Modelling}, 49:2213--2223, 2009.

\bibitem{cheng15}
X.~Cheng, R.~Gong, and W.~Han.
\newblock A new {K}ohn-{V}ogelius type formulation for inverse source problems.
\newblock {\em Inverse Problems and Imaging}, 9(4):1051--1067, 2015.

\bibitem{Bad00}
A.~El~Badia and T.~Ha-Duong.
\newblock An inverse source problem in potential analysis.
\newblock {\em Inverse Problems}, 16:651--663, 2000.

\bibitem{Elv20}
O.~L. {Elvetun} and B.~F. {Nielsen}.
\newblock {A regularization operator for source identification for elliptic
  PDEs}.
\newblock {\em Accepted for publication in Inverse Problems and Imaging. Also
  available as arXiv e-prints}, page arXiv:2005.09444, May 2020.

\bibitem{Han11}
M.~Hanke and W.~Rundell.
\newblock On rational approximation methods for inverse source problems.
\newblock {\em Inverse Problems and Imaging}, 5(1):185–202, 2011.

\bibitem{Het96}
F.~Hettlich and W.~Rundell.
\newblock Iterative methods for the reconstruction of an inverse potential
  problem.
\newblock {\em Inverse Problems}, 12:251--266, 1996.

\bibitem{hinze19}
M.~Hinze, B.~Hofmann, and T.~N.~T. Quyen.
\newblock A regularization approach for an inverse source problem in elliptic
  systems from single {C}auchy data.
\newblock {\em Numerical Functional Analysis and Optimization},
  40(9):1080--1112, 2019.

\bibitem{BIsa05}
V.~Isakov.
\newblock {\em Inverse Problems for Partial Differential Equations}.
\newblock Springer-Verlag, 2005.

\bibitem{kun94}
K.~Kunisch and X.~Pan.
\newblock Estimation of interfaces from boundary measurements.
\newblock {\em SIAM J. Control Optim.}, 32(6):1643–1674, 1994.

\bibitem{ring95}
W.~Ring.
\newblock Identification of a core from boundary data.
\newblock {\em SIAM Journal on Applied Mathematics}, 55(3):677--706, 1995.

\bibitem{song12}
S.~J. Song and J.~G. Huang.
\newblock Solving an inverse problem from bioluminescence tomography by
  minimizing an energy-like functional.
\newblock {\em J. Comput. Anal. Appl.}, 14:544--558, 2012.

\bibitem{Wan17}
X.~Wang, Y.~Guo, D.~Zhang, and H.~Liu.
\newblock Fourier method for recovering acoustic sources from multi-frequency
  far-field data.
\newblock {\em Inverse Problems}, 33(3), 2017.

\bibitem{Zha18}
D.~Zhang, Y.~Guo, J.~Li, and H.~Liu.
\newblock Retrieval of acoustic sources from multi-frequency phaseless data.
\newblock {\em Inverse Problems}, 34(9), 2018.

\bibitem{Zha19}
D.~Zhang, Y.~Guo, J.~Li, and H.~Liu.
\newblock Locating multiple multipolar acoustic sources using the direct
  sampling method.
\newblock {\em Communications in Computational Physics}, 25(5):1328--1356,
  2019.

\end{thebibliography}

\end{document}